\documentclass[12pt,reqno]{amsart}
\usepackage{amsthm}
\usepackage{amssymb}
\usepackage{amsmath}
\usepackage{mathrsfs}
\usepackage{amsfonts}
\usepackage{amssymb,amsmath}
 \usepackage[colorlinks, linkcolor=black, citecolor=blue,   hypertexnames=false]{hyperref}
 \usepackage[numbers,sort&compress]{natbib}
\newtheorem{proposition}{Proposition}[section]
\newtheorem{theorem}[proposition]{Theorem}
\newtheorem{corollary}[proposition]{Corollary}
\newtheorem{lemma}[proposition]{Lemma}
\theoremstyle{definition}

\newtheorem{remark}[proposition]{Remark}

\numberwithin{equation}{section}

\def \no#1#2#3 {{\bf #1} (#3), #2.}
\def \eds#1#2#3 {#1, #2, #3.}

\title[Global stability of planes under SMCF ]
{ \bf
Global transversal stability of Euclidean planes under skew mean curvature flow evolutions }

\author [Z. Li ]
{Ze Li }

\address{Ze Li
\newline\indent
School of Mathematics and Statistics, Ningbo University
\newline\indent
Ningbo, 315211, P.R. China
}
\email{rikudosennin@163.com}


\begin{document}

\begin{abstract}
In this paper, we prove that 2 dimensional transversal small perturbations of d-dimensional Euclidean planes under the skew mean curvature flow lead to global solutions which converge  to the unperturbed planes in suitable norms. And we clarify the long time behaviors of the solutions in Sobolev spaces.
\end{abstract}

 \maketitle

\section{Introduction}

Let  $\Sigma$  be a $d$-dimensional oriented manifold and $(\mathcal{N},h)$ be a  $(d+2)$-dimensional oriented Riemannian manifold. Assume that  ${\Bbb I}$
is an interval containing $t=0$, and $F :{\Bbb I}\times \Sigma \to \mathcal{N}$ is a family of immersions. For each given $t\in {\Bbb I}$,
denote $\Sigma_t= F(t,\Sigma)$  the submanifold and  $\mathbf{H}(F)$ its mean curvature vector. Denote the tangent bundle of the submanifold by $T\Sigma_t$ and the   normal bundle by $N\Sigma_t$ respectively. There exists a natural induced complex structure $J(F)$  for $N\Sigma_t$ (rank two) via  simply rotating a vector in the normal space
by $\frac{\pi}{2}$ positively. To be  precise, for any point $z=F(t,x) \in\Sigma_ t$ and normal vector
$\nu\in N_{z} \Sigma_ t$, define $J(F)$ by letting  $J(F)\nu\bot \nu=0$ and $\omega(F _*(e_1 ),...,F_* (e_d ),\nu,J(F)\nu) > 0$, where
$\omega$ is the volume form of $\mathcal{N}$ and $\{e_ 1 ,...,e_d\}$ is an oriented basis of $T\Sigma_t$. We shall call
the binormal vector $J(F)\mathbf{H}(F)\in N\Sigma_t$  the skew mean curvature vector. And the skew mean curvature flow (SMCF) for immersions $F:{\Bbb I}\times \Sigma\to \mathcal{N}$ is
defined by
\begin{align}\label{yu}
\left\{
  \begin{array}{ll}
     {\partial_t}F= J(F)\mathbf{H}(F), & \hbox{ } \\
    F(0,x)=F_0(x),\mbox{ }x\in \Sigma. & \hbox{}
  \end{array}
\right.
\end{align}

In view of geometric flows, the SMCF evolves a codimension 2 submanifold along its binormal direction with a speed given by its
mean curvature. It is remarkable that SMCF was historically derived from both physics and  mathematics. The  motivations from physics  are  vortex filament equations  (VFE) ,  the localized
induction approximation (LIA) of the hydrodynamical Euler equations, and  asymptotic dynamics of vortices in
 superfluidity and superconductivity.

In fact,  the 1-dimensional SMCF in the Euclidean space $\Bbb R^3$ is the vortex filament equation
\begin{align}
\partial_t u= \partial_su\times  u,\mbox{ }u:(s,t)\in \Bbb R\times \Bbb R\longmapsto u(s,t)\in \Bbb R^3,
\end{align}
where $t$ denotes time, $s$ denotes the arc-length parameter of the curve $u(t,\cdot)$,  and $\times$ denotes
the cross product in $\Bbb R^3$. The VFE describes  the free motion
of a vortex filament, see Da Rios \cite{R} and Hasimoto \cite{Ha}.

The SMCF also appears in the study of asymptotic dynamics of vortices in
the context of superfluidity and superconductivity. For the Gross-Pitaevskii equation, which models  Bose-Einstein condensates,
physicists conjecture  that the vortices would evolve along the SMCF. This
was first verified by Lin \cite{L1} for the vortex filaments in $\Bbb R^3$. Similar phenomena were also  observed
for other  PDEs, for example,
 it was shown for the Ginzburg-Landau heat flow that the energy asymptotically concentrates on the codimension 2 vortices whose motion is governed by
the mean curvature flow.

The other motivation is the LIA of  hydrodynamical Euler equations, which describes the limit of
a generalized  Biot-Savart formula.  In fact, let $M\subset \Bbb R^{d}$  with $d\ge 3$ be a closed
oriented submanifold of codimension 2. Consider the vorticity 2-form
$\eta_{M}$ supported on this submanifold: $\eta_{M}=C\cdot\delta_{M}$. We shall call $M$ a
higher dimensional vortex filament or membrane. Then for any dimension $d\ge 3$ the divergence-free vector
 field $v$ in $\Bbb R^d$  satisfying $curl v = \eta_{M}$  is given by a generalized Biot-Savart formula,
which only holds for points away from $M$. In order to derive the formula of $v$ for points within $M$,
 one defines  truncation vector field $v_{\epsilon}$ by the above generalized Biot-Savart formula and some truncation.
 It was shown that the limit of $(\ln \epsilon)^{-1}v_{\epsilon}(x)$ for $x\in M$ as $\epsilon\to 0$ is the skew mean curvature
 vector of $M \hookrightarrow \Bbb R^{d}$. Hence,
the LIA approximation for a vortex membrane (or higher filament) $M$ in $\Bbb R^d$
up to a suitable scaling coincides with the SMCF. These facts were discovered  by  Shashikanth \cite{Sh} and generalized by  Khesin \cite{Kh}.

SMCF  also emerges in different mathematical problems. In study of  nonlinear Grassmannians, Haller-Vizman \cite{HV}
noted that SMCF is the Hamiltonian flow  of the volume functional on the space  consisting of all co-dimensional 2
immersions of a given Riemannian manifold. Notice that this space admits a generalized Marsden-Weinstein symplectic structure \cite{MW} which provides the volume form.
The SMCF of surfaces in $\Bbb R^4$ is included by a vast energy conserved motion project raised by Lin and his collaborators \cite{L2}. Terng \cite{Te} also proposed SMCF under the name of star MCF.

Let us describe the non-exhaustive list of works on SMCF. The case $d=1$, i.e. the  VFE, has been intensively studied by many authors from many views. We recommend the reader to read the survey Vega \cite{V} and Gomez's thesis \cite{G} for VFE.  The works on the case $d\ge 2$ are much less.  Song-Sun \cite{Song1}  proved local existence of SMCF for $F:\Sigma \to \Bbb R^{4}$ with compact oriented surface $\Sigma$. This was generalized by Song \cite{Song2} to $F:\Sigma \to \Bbb R^{d+2}$ with compact oriented surface $\Sigma$ for all $d\ge 2$.
Khesin-Yang \cite{KY} constructed an example showing that the SMCF can blow up in finite time. Jerrard proposed a notion of weak solutions to the SMCF in \cite{Je}.
 Song  \cite{Song3} proved that the Gauss map of a $d$ dimensional SMCF in $\Bbb R^{d+2}$ satisfies a Schr\"odinger map flow equation (see e.g. \cite{UT,DW,CSU,BIKT} ).

For $d\ge 2$, the global existence theory of SMCF is largely open even for small data. This is what we aim to solve in this paper. Let us consider the case $\Sigma=\Bbb R^d$, $\mathcal{N}=\Bbb R^{d+2}$ in SMCF. It is easy to see $F(t,x_1,...,x_d)=(x_1,...,x_d,0,0)$ is a solution of SMCF, i.e. the d-dimensional planes.
We consider the transversal perturbations of this plane, i.e., we seek for a graph like solution of the form
\begin{align}\label{Ffgm}
F(t,x)=(x_1,...,x_d,u_1(t,x),u_2(t,x)).
\end{align}
Our following main theorem states that if the initial perturbation $u_1(0,\cdot),u_2(0,\cdot)$ is sufficiently small, then SMCF has a global graph like solution and the asymptotic behaviors can be clearly determined.

In fact, we have
\begin{theorem}\label{th2}
Given $d\ge 2$, let $k$ be the smallest integer such that  $k>max({\frac{1}{2}(d+7),d+1})$. Assume that $u^0_1, u^0_2:\Bbb R^d\to \Bbb R$ are
functions belonging to  $H^{\sigma}$ for any $\sigma\ge 0$. Given $q\in (1,2)$ satisfying (\ref{huashan}),
there exists a sufficiently small constant $\epsilon>0$  such that if
\begin{align}\label{yuxx78}
\|u^0_1\|_{ W^{k,2}\cap W^{2,q}}+\|u^0_2\|_{W^{k,2}\cap W^{2,q}}\le \epsilon,
\end{align}
then there exists a  unique global smooth solution to SMCF of the form (\ref{Ffgm}) such that $u_1(t)=u^0_1$, $u_2(t)=u^0_2$ when $t=0$. Moreover, one has
\begin{align}\label{result}
\langle t\rangle^{\frac{d}{2}(\frac{2}{q}-1)}(\|u_1(t)\|_{W^{2,\frac{q}{q-1}}_x}+\|u_2(t)\|_{W^{2,\frac{q}{q-1}}_x})+  \|u_1(t)\|_{H^{k}_x}+\|u_2(t)\|_{H^k_x}\lesssim  \epsilon
\end{align}
for any $t\in\Bbb R$. And there exist time independent complex valued  functions $\phi_{\pm}\in H^2$ so that
\begin{align}\label{result3}
\lim\limits_{t\to \pm \infty}\|u_1(t)+ \mathrm{i}u_2(t) -e^{\mathrm{i}t\Delta}\phi_{\pm} \|_{H^{2}_x}=0.
\end{align}
\end{theorem}

Theorem \ref{th2} also shows the $d$-dimensional plane is transversally stable under the SMCF evolution. Let us describe the main idea in the proof of  Theorem \ref{th2}. The most challenging  problem in dealing with  SMCF is that it is highly quasilinear and the leading part is degenerate along the tangent bundle. In the graph like form (\ref{Ffgm}), the degenerateness can be avoided using a suitable equivalent formulation of SMFC. The quasilinear nature seems to be unavoidable as far as we know. In fact,  whether there exists a gauge transform to  make SMCF semilinear is largely open, see Khesin-Yang \cite{KY} for some discussions.   In this paper, we adopt the strategy  dating back to Klainerman \cite{Kl1}, the idea is that dispersive estimates of linear part provide time decay of solutions in $L^{p}$ norms with $p>2$ and the high order energy estimates give a chance to overcome the derivative loss in the nonlinear parts. The combination of dispersive estimates and energy estimates can close the bootstrap in the small data case, and thus finishing the proof. The technical part is to choose suitable working spaces in the bootstrap argument.

The local well-posedness also requires some efforts.  In fact, \cite{Song2} raised   local well-posedness for non-compact manifold $\Sigma$ as an open problem. The proof of \cite{Song1} used compactness of second fundamental forms corresponding to  a family of immersed  manifolds with uniform volumes, which is unavailable for general non-compact  manifolds. In the $d=2$ and small data case, \cite{Song1}'s argument indeed works for graph like solutions considered here. But for $d\ge 3$, we need some refinements, because the second fundamental form generally cannot control the graph function $u$ in $d\ge 3$. (See Sec. 3 for more discussions.) In this work, for the local Cauchy problem we can indeed deal with arbitrary large data, see Theorem \ref{2K} in Section 3.
In the view of PDEs, Theorem \ref{2K} is a local existence and uniqueness theorem for quasilinear Schr\"odinger equations. There are many works in  local well-posedness theory on  general quasilinear Schr\"odinger equations, see for instance the pioneering works by Kenig-Ponce-Vega \cite{Ke2,Ke3,Ke4}, Kenig-Ponce-Rolvung-Vega \cite{Ke1}, Marzuola-Metcalfe-Tataru \cite{MMT1,MMT,MMT3}. Generally, local well-posedness of quasilinear Schr\"odinger equations with large data needs additional non-trapping data conditions. Here, due to the  geometric structure of SMCF, no non-trapping conditions are needed. This advantage is mainly caused by  the high order derivative energy estimates enjoyed by the second fundamental form. In addition, we remark that the well-posedness in the second part of Theorem \ref{2K} includes    the existence of
a local solution, uniqueness, and continuous dependence on the initial data.

The paper is organized as follows. In Section 2, we reduce SMCF to a quasilinear Schr\"odinger equation. In Section 3, we present the local  well-posedness and arbitrary order energy estimates. In Section 4, we prove the main theorem for $d\ge 2$. In Section 5, we finish the whole proof.

\noindent{\bf Notations.}
The notation $A\lesssim B$ means there exists some universal constant $C>0$ such that $A\le CB$.
The notation $C_{\gamma_1,...,\gamma_n}$ denotes some constant depending on the parameters $\gamma_1,...,\gamma_n$, and generally it varies from line to line.
Denote $\langle x\rangle =\sqrt{1+|x|^2}$.
The Fourier transform $f\mapsto \widehat{f}$ is denoted by
\begin{align}
\widehat{f}(\xi)=\int_{\Bbb R^d} f(x)e^{-i\xi\cdot x}dx.
\end{align}
The usual Sobolev spaces $H^s$  are defined by
\begin{align}
\|f\|_{H^{s}}=\|\langle \xi\rangle^{s}\widehat{f}(\xi)\|_{L^2_{\xi}}.
\end{align}
Let $H^{\infty}=\cap^{\infty}_{s=0}H^s$.

\section{Master equation}

Let $F:\Sigma\to \Bbb R^{d+2}$, and $\mathbf{H}$ be the mean curvature vector. Then
\begin{align*}
\Delta_g F={\mathbf{H}},
\end{align*}
where $\Delta_g$ denotes the Laplacian on $\Sigma$ of the induced metric $g$ on $T\Sigma_t$ given by
\begin{align*}
g_{ij} =\partial_{x_i}F\cdot\partial_{x_j}F.
\end{align*}
Since one has
\begin{align*}
\Delta_g F^{\alpha}=g^{ij}(\frac{\partial^2}{\partial x_i\partial x_j}F^{\alpha}-\Gamma^{l}_{ij}\frac{\partial }{\partial x_l}F^{\alpha}),
\end{align*}
and the vector $g^{ij}\Gamma^{l}_{ij}\frac{\partial }{\partial x_l}F$ belongs to $T\Sigma_t$, we see
\begin{align*}
\mathbf{H}=(\Delta_g F)^{\bot}=\sum_{l=1,2}(g^{ij} \frac{\partial^2 F}{\partial x_i\partial x_j}\cdot\nu_l)\nu_l,
\end{align*}
where  $\{\nu_{l}\}^{2}_{l=1}$ denotes the orthonormal basis  of $N\Sigma_t$ such that $J(\nu_1)=\nu_2$, $J(\nu_2)=-\nu_1$.  Then, the SMCF can be written as
\begin{align*}
\partial_{t}F = (g^{ij} \frac{\partial^2 F}{\partial x_i\partial x_j }\cdot\nu_1)\nu_2-(g^{ij} \frac{\partial^2 F}{\partial x_i\partial x_j}\cdot\nu_2)\nu_1.
\end{align*}

Now, let us consider the case when $F$ is represented by a graph, i.e., $F(x,t)=(x_1,...,x_d,u_1,u_2)$, where $u_1,u_2$ are functions of $x,t$. It is easy to see
the $(F^1,...,F^d)$ components of this graph like map $F$ satisfies the afore $d$-equations of SMCF. Set the orthonormal basis $\{\nu_1,\nu_2\}$  of $N\Sigma_t$ to be
\begin{align*}
\nu_{1}&=\frac{1}{\sqrt{1+|\partial_xu_1|^2}}(\partial_xu_1,-1,0) \\
\nu_{2}&=\frac{1}{|\widetilde{\nu}_2|}\widetilde{\nu}_2\\
\widetilde{\nu}_2&= (\partial_{x}u_2,0,-1) - [(\partial_{x}u_2,0,-1)\cdot\nu_1]\nu_1,
\end{align*}
where we denote $\partial_x u=(\partial_{x_1}u,..., \partial_{x_n}u)$, and $|\cdot|$ denotes the Euclidean norm in $\Bbb R^{d}$ or $\Bbb R^{d+2}$. Further calculations give
\begin{align*}
\nu_{2}&=\frac{1}{\Lambda} (\partial_x u_2-\frac{\partial_{x}u_2\cdot\partial_x u_1}{1+|\partial_x u_1|^2}\partial_x u_1, \frac{\partial_x u_1\cdot \partial_x u_2}{1+|\partial_x u_1|^2}, -1)\\
\Lambda&=\left(|\partial_x u_2-\frac{\partial_{x}u_2\cdot\partial_x u_1}{1+|\partial_x u_1|^2}\partial_x u_1|^2+|\frac{\partial_x u_1\cdot \partial_x u_2}{1+|\partial_x u_1|^2}|^2+1\right)^{\frac{1}{2}}.
\end{align*}

Therefore,  the SMCF reduces  to
\begin{align}\label{sml}
\left\{
  \begin{array}{ll}
    \partial_t u_1&= g^{ij}(\frac{\partial^2 F}{\partial x_i\partial x_j  }\cdot \nu_2 )\frac{1}{\sqrt{1+|\partial_x u_1|^2}} + g^{ij}(\frac{\partial^2 F}{\partial x_i\partial x_j }\cdot \nu_1 )\frac{\partial_x u_1\cdot \partial_x u_2}{(1+|\partial_x u_1|^2)\Lambda} \hbox{ } \\
    \partial_t u_2&=- g^{ij}(\frac{\partial^2 F}{\partial x_i\partial x_j }\cdot \nu_1 )\frac{1}{\Lambda}. \hbox{ }
  \end{array}
\right.
\end{align}
To clarify the main linear part of the above equation, we calculate the expansions of $\nu_1,\nu_2,\Lambda$.
We observe that when $|\nabla u_1|+|\nabla u_2|$ is sufficiently small, $\nu_1,\nu_2,\Lambda$ have the expansions
\begin{align*}
\nu_1&=(\partial_x u_1,-1,0)[1-\frac{1}{2}|\partial_x u_1|^2+O(|\partial_x u|^4)]\\
\Lambda&=1+O(|\partial_x u|^2)\\
\nu_2&=(\partial_x u_2-\frac{\partial_{x}u_2\cdot\partial_x u_1}{1+|\partial_x u_1|^2}\partial_x u_1, \frac{\partial_x u_1\cdot \partial_x u_2}{1+|\partial_x u_1|^2}, -1)[1+O(|\partial_x u|^2)].
\end{align*}
Thus one has
\begin{align*}
\frac{\partial^2F}{\partial x_i\partial x_j  }\cdot \nu_2=-\frac{\partial^2u_2}{\partial x_i\partial x_j  }+ O({\partial^2_xu }|\partial_xu|^2)\\
\frac{\partial^2F}{\partial x_i\partial x_j  }\cdot \nu_1=-\frac{\partial^2u_1}{\partial x_i\partial x_j  }+ O({\partial^2_xu }|\partial_xu|^2).
\end{align*}
Now, we see (\ref{sml}) reduces to
\begin{align}\label{smm2}
\left\{
  \begin{array}{ll}
    \partial_t u_1&= -g^{ij}\frac{\partial^2u_2}{\partial x_i\partial x_j  }+ O(g^{ij}{\partial^2_{ij}u }|\partial_xu|^2) \hbox{ } \\
    \partial_t u_2&= g^{ij}\frac{\partial^2u_1}{\partial x_i\partial x_j }+O(g^{ij}{\partial^2_{ij}u }|\partial_xu|^2). \hbox{ }
  \end{array}
\right.
\end{align}
Let us calculate the expansion of $g^{ij}$. In fact, since
\begin{align}\label{Hjk}
g_{ij} =\partial_{x_i}F\cdot\partial_{x_j}F =\delta_{ij}+\partial_{i}u\cdot \partial_{j}u,
\end{align}
one has by straightforward calculations  that
\begin{align*}
g^{ij}=\delta_{ij}-\frac{\partial_{x_i}u\cdot \partial_{j}u}{1+|\partial_xu|^2}=\delta_{ij}+O(|\partial_x u|^2).
\end{align*}
Therefore,   (\ref{smm2}) can be written as
\begin{align}\label{smk3}
\left\{
  \begin{array}{ll}
    \partial_t u_1&=  -\Delta u_2  + O({\partial^2_xu }|\partial_xu|^2) \hbox{ } \\
    \partial_t u_2&=   \Delta u_1  + O({\partial^2_xu }|\partial_xu|^2), \hbox{ }
  \end{array}
\right.
\end{align}
where the $O({\partial^2_xu }|\partial_xu|^2)$ in fact contains many terms including the  leading order cubic term like ${\partial^2_xu }|\partial_xu|^2$ and remainder terms of higher powers of $|\partial_x u|$.

Let $\phi=u_1+ \mathrm{{i}}u_2$. We see  (\ref{smk3}) is indeed a quasilinear Schr\"odinger equation with at least cubic interactions:
\begin{align}\label{mss}
 \mathrm{i}\partial_t \phi+\Delta \phi= O({\partial^2_x\phi }|\partial_x\phi|^2),
\end{align}
where the RHS of (\ref{mss}) can be written as
\begin{align}\label{ddmss}
\sum c_{iji'j'} \partial^2_{x_ix_j}\phi^{\pm}  \partial_{x_{i'}}\phi^{\pm}\partial_{x_{j'}}\phi^{\pm}+\mathcal{R}.
\end{align}
Here, $c_{iji'j'}$ are universal constants,  $\phi^{+}$ and $\phi^{-}$ denote $\phi$ and $\bar{\phi}$ respectively, and the remainder  $\mathcal{R}$ point-wisely satisfies
\begin{align}\label{mq}
|\partial^{l}_x \mathcal{R}|\le C_{l}\sum_{l_1+...+l_j\le l,4\le j\le 2^{10l} }|\partial^{l_1}_x\partial^2_{x}\phi^{\pm}|...|\partial^{l_j}_x\partial_x\phi^{\pm}|
\end{align}
for any $l\ge 0$ provided that $\|\partial_x\phi^{\pm}\|_{L^{\infty}_{t,x}}$ is sufficiently small.

Besides the above approximate equations near 0, one also has the following exact equation,
\begin{align*}
\left\{
  \begin{array}{ll}
    \partial_t u_1&=  -\frac{1}{\Lambda  \sqrt{1+|\partial_x u_1|^2}}g^{ij} \partial^2_{x_ix_j}u_2  \hbox{ } \\
    \partial_t u_2&=   \frac{1}{\Lambda  \sqrt{1+|\partial_x u_1|^2}}g^{ij}\partial^2_{x_ix_j} u_1, \hbox{ }
  \end{array}
\right.
\end{align*}
which further gives
\begin{align}\label{H3M}
i \partial_t\phi=\frac{1}{\Lambda  \sqrt{1+|\partial_x u_1|^2}}g^{ij} \partial^2_{x_ix_j}\phi.
\end{align}

Generally, a further detailed calculation reveals that (\ref{mss})-(\ref{mq})  hold as well provided that  $\|\partial_x\phi^{\pm}\|_{L^{\infty}_{t,x}}$ is finite. In fact, we have
\begin{remark}\label{12}
 If $\|\partial_x\phi^{\pm}\|_{L^{\infty}_{t,x}}<\infty$, then given $l\ge 0$ there exists a constant $C>0$ depending only on $l$  such that (\ref{mss}), (\ref{ddmss}) and (\ref{mq}) hold.
\end{remark}
\section{Local well-posedness and Energy estimates}

Suppose $F :{\Bbb I}\times \Sigma\to  \Bbb R^{d+2}$ is a solution to the SMCF.  Recall that for each $t\in \Bbb I$, $g=g(t)$
denotes the induced metric on $\Sigma$. Denote the associated volume form on $\Sigma$
by $d\mu=d\mu(t)$.  Consider the pullback bundle $F^*T \Bbb R^{d+2}$ defined   over the base manifold $\Bbb I\times \Sigma$. This bundle splits naturally into  the ``spatial'' subbundle $\mathfrak{{H}}$ and the normal bundle $\mathfrak{N}$. Pulling back the metric and connection of $\Bbb R^{d+2}$ naturally  induces a metric $g^{\mathfrak{N}}$ and connection $\nabla^{\mathfrak{N} }$ defined on the bundle $\mathfrak{N}$, see \cite{Ba} for a detailed presentation. For simplicity, we write $\nabla$  instead of  $\nabla^{\mathfrak{N} }$.

The energy estimates are due to Song-Sun \cite{Song1} and Song \cite{Song2}. We recall the results in the following  lemma.

Let $F:\Bbb I\times \Sigma\to \Bbb R^{d+m}$ be a map given by
\begin{align}
F(t,x_1,...,x_d)=(x_1,...,x_d,u_1(t,x_1,..,x_d),...,u_{m}(t,x_1,...,x_d)).
\end{align}
Let $\mathbf{A}$ denote the corresponding second fundamental of $\Sigma_t:=Graph(u)$.
Given an integer $l\ge 0$ and a number $p\in [1,\infty)$, define the Sobolev norm of ${\bf A}$ by
\begin{align}\label{hhh1}
\|{\bf A}\|_{H^{l,p}}=\left(\sum^{l}_{k=0}\int_{\Sigma} |\nabla^{l} {\bf A}|^p_{g}d\mu\right)^{\frac{1}{p}}.
\end{align}
If $p=\infty$, we define $\|{\bf A}\|_{H^{l,\infty}}$ by
\begin{align*}
\|{\bf A}\|_{H^{l,\infty}}=\sum^{l}_{k=0}\||\nabla^{l} {\bf A}|_{g}\|_{L^{\infty}}.
\end{align*}

And define the usual Sobolev norm of the Hessian  $D^2 u$ of $u$ by
\begin{align}\label{hhh}
\|D^2 u\|_{W^{l,p}}=\left(\sum^{l}_{k=0}\int_{\Sigma} |D^{l} D^2u|^pdx\right)^{\frac{1}{p}}.
\end{align}

\begin{lemma}[\cite{Song1,Song2}]\label{e}
With above notions and notations, considering $F:\Bbb I\times \Bbb R^d\to \Bbb R^{d+2}$, there holds point-wisely that
\begin{align}
|{\bf A}|^2_g&\le|D^2 u|^2\le (1+|Du|^2)^3|{\bf A}|^2_g\label{e1},
\end{align}
and for any given $l\ge 0$ there exists a polynomial $Q_{l}$ depending only on $l$ such that
\begin{align}
|\nabla^{l}{\bf A}|_g&\le|D^{l+2} u|+Q_{l}(|Du|)\sum|D^{l_1+1}u|...|D^{l_j+1}u|\label{e2s}\\
|D^{l+2}u|&\le (1+|Du|^2)^{\frac{l+3}{2}}|\nabla^2{\bf A}|_g  +Q_{l}(|Du|)\sum|D^{l_1+1}u|...|D^{l_j+1}u|,\label{e3s}
\end{align}
where the summations are taken over the indices $(l_1,...,l_j)$ satisfying
\begin{align*}
l_1+...+l_j=l+1, \mbox{ }l_1\ge \l_2\ge ...\ge l_j, \mbox{ } l\ge l_i\ge 1.
\end{align*}
Let $|Du|\le \alpha$, then for any $l\ge 0$,
\begin{align}
\|{\bf A}\|_{H^{l,2}}\le C_{\alpha} \sum^{l+1}_{k=1}\|D^2u\|^{k}_{W^{l,2}}\label{e2}.
\end{align}
Let $|Du|\le \alpha$,  $|D^2u|\le \beta$, then for any $l\ge 0$,
\begin{align}\label{e3}
\|D^2u\|_{W^{l,2}}\le C_{\alpha,\beta}( \sum^{l}_{k=1} \|{\bf A}\|^{k}_{H^{l,2}}+\sum^{l}_{k=2} \|D^2u\|^{k}_{W^{l,2}}).
\end{align}
The smooth solution $F(t,x)$ of SMCF satisfies
\begin{align}
\frac{d}{dt}\int_{\Sigma }|\nabla^{l} {\bf A} |^2_{g}d\mu&\le C \max_{\Sigma }|{\bf A}|^2_{g} \int_{\Sigma}|\nabla^{l} {\bf A} |^2_{g}d\mu,\label{e4}
\end{align}
where $C>0$ is a universal constant depending only on $d,l$.
Given $p\ge 2$, the smooth solution $F(t,x)$ of SMCF satisfies
\begin{align}\label{Maz1}
\frac{d}{dt}\int_{\Sigma }| {\bf A} |^p_{g}d\mu&\le C (1+\max_{\Sigma }|{\bf A}|^2_{g}) \int_{\Sigma}|  {\bf A} |^p_{g}d\mu+C\int_{\Sigma}| \nabla {\bf A} |^p_{g}d\mu,
\end{align}
where $C>0$ is a universal constant depending only on $d,p$. Along the SMCF, the induced  metric $g$ with $g_{ij}=\partial_{x_i}F\cdot\partial_{x_j} F$ satisfies
\begin{align}\label{Mazw}
 {\partial}_tg_{ij}=-2\langle J{\bf H},{\bf A}(\partial_{x_i}F,\partial_{x_j}F)\rangle,
\end{align}
and the associated volume form $\mu(t)$ satisfies $\partial_t\mu(t)=0$.
\end{lemma}

We remark that although (\ref{e3}) was not directly stated in \cite{Song1,Song2},  it indeed  follows by (\ref{e3s}). In $d=2$, \cite{Song1} proved a stronger result:
\begin{align}\label{e3p}
\|D^2u\|_{W^{l,2}}\le C_{\alpha,\beta} \sum^{l}_{k=1} \|{\bf A}\|^{k}_{H^{l,2}}.
\end{align}
 for any $l\ge 0$, if  $|Du|\le \alpha$,  $|D^2u|\le \beta$. But in $d\ge 3$ one only has  (\ref{e3}), and  (\ref{e3p}) fails. But   (\ref{e3})  suffices to work in the small data case by noting that the RHS of (\ref{e3}) is at least quadratic in $\|D^2u\|_{W^{l,2}}$.

We also remark that although (\ref{e1}), (\ref{e2}),  (\ref{e4}) were proved in \cite{Song1,Song2} for compact $\Sigma$, its proof works for general manifolds where Gagliardo-Nirenberg inequalities hold.

\subsection{Some Sobolev inequalities of tensors}

The following two lemmas are useful for graph solutions.
\begin{lemma}\label{hbnjkmkk}
Let  $u:\Bbb R^d\to \Bbb R^{d+2}$ be a graph, $u(x)=(x_1,...,x_d,u^1(x),u^2(x))$. Suppose that
\begin{align*}
\|Du\|_{L^{\infty}}\le \alpha,
\end{align*}
and the corresponding second fundamental form ${\bf A}$ satisfies
\begin{align*}
\|{\bf A}\|_{H^{0,\infty}}=\beta<\infty, \mbox{ } \|{\bf A}\|_{H^{k,2}}=B_{k}<\infty
\end{align*}
for any $k\in \Bbb Z_+$. Then given $k\in\Bbb Z_+$, $p\in [2,\infty)$, there exist an integer $K\in\Bbb Z_+$ depending only on $k,d,p$, and a polynomial $P_{K}$ such that
\begin{align}\label{Vbnm}
\|D^2u\|_{W^{k,p}}\lesssim_{\beta,\alpha} P_{K}(B_{K}).
\end{align}
\end{lemma}
\begin{proof}
Since $\|Du\|_{L^{\infty}_x }\le  \alpha$,
the volume $\mu(\cdot)$ induced by immersion $u$ and the usual volume $|\cdot|$ in Euclidean space $\Bbb R^d$ satisfy
\begin{align}\label{Bnhnm}
|O|\lesssim_{d,\alpha} \mu(O)\lesssim_{d,\alpha}  |O|
\end{align}
for any measurable set $O$ in $\Bbb R^d$.
Given $l\ge 1$, $p\in[2,\infty)$, integrating (\ref{e3s}) in $\Bbb R^d$ we infer from (\ref{Bnhnm}) that
\begin{align}
\|D^{l+2}u\|_{L^p}\lesssim \|{\bf A}\|_{H^{l,p}} + C(\alpha)\sum \||D^{l_1+1}u|...|D^{l_j+1}u|\|_{L^p },\label{ffe3s}
\end{align}
where the summations are taken over the indices $(l_1,...,l_j)$ satisfying
\begin{align*}
l_1+...+l_j=l+1, \mbox{ }l_1\ge \l_2\ge ...\ge l_j, \mbox{ } l\ge l_i\ge 1.
\end{align*}
Let's begin with $\|D^{2}u\|_{L^2\cap L^{\infty}}$. By (\ref{e1}), one has
\begin{align}\label{1cvb}
\|D^2u\|_{L^2\cap L^{\infty}}\lesssim_{\alpha}\|{\bf A}\|_{H^{0,2} \cap H^{0,\infty} }\le_{\alpha} B_0+\beta.
\end{align}
The other high derivative terms need more efforts.

Applying Hamilton's interpolation inequality (see Lemma \ref{Hami}) to ${\bf A}$ shows
\begin{align*}
\int_{\Sigma} |\nabla^{j} {\bf A}|^{\frac{2i}{j}}d\mu\le C\max_{\Sigma}|{\bf A}|^{2(\frac{i}{j}-1)}\int_{\Sigma} |\nabla^{i} {\bf A}|^{2}d\mu,
\end{align*}
which implies
\begin{align}\label{ghbn}
\|{\bf A}\|_{H^{j,\frac{2i}{j}}}\le C(\alpha,\beta,B_i)
\end{align}
for any $1\le j\le i$.

Given  any $p\in[2,\infty)$, taking $l=1$ in (\ref{ffe3s}) we find
\begin{align}
\|D^{3}u\|_{L^p_x}\lesssim  \|{\bf A}\|_{H^{3,p}} +  C(\alpha)\||D^2u||D^2u|\|_{L^p},
\end{align}
which,  together with  (\ref{1cvb}) and (\ref{ghbn}),  further gives
\begin{align}
\|D^{3}u\|_{L^p_x}\lesssim_{\alpha,\beta}  C(B_{[\frac{3p}{2}]+1}) + C(B_0).
\end{align}
Now, let us prove (\ref{Vbnm}) by induction. If for some $k=\nu\ge 1$  (\ref{Vbnm}) has been verified, let us prove   (\ref{Vbnm}) for $k=\nu+1$.
For this, taking $l=\nu+1$ in (\ref{ffe3s}) we see the summation indexes $\{l_1,...,l_j\}$ in the   RHS of  (\ref{ffe3s}) are all   smaller than or equal to $\nu+1$.
Thus the RHS of (\ref{ffe3s}) only involves derivatives of $u$ of order lower than or equal to $\nu+2$. Then by induction assumption,  (\ref{Vbnm}) also holds in the $k=\nu+1$ case.
\end{proof}

\begin{remark}
In the above proof, usual Gagliardo-Nirenberg  inequalities cannot be applied to the second fundamental form  ${\bf A}$,
since the embedding constants may depend on the metric $g$ which is also $t$ dependent. So one takes Hamilton's interpolation inequality which is free of metrics.
\end{remark}

\begin{lemma}\label{Maz}
Let $u:\Bbb R^d\to \Bbb R^{d+2}$ be a graph, $u(x)=(x_1,...,x_d,u^1(x),u^2(x))$. Denote $B_r(y)$ to be the ball of radius $r>0$  in $\Bbb R^{d+2}$ centered at $y\in\Bbb R^{d+2}$. Let $\Sigma=u(\Bbb R^{d})$, and denote the induced metric, measure, the second fundamental form, the mean curvature vector  associated with  the immersion  $u$ by $g$, $\mu$, ${\bf A}$, and ${\bf H}$ respectively. View  $\mu$ as a measure in $\Bbb R^{d+2}$ with support in $\Sigma$. If
$
\mu(B_1(y)\bigcap \Sigma)+\|{\bf H}\|_{L^{d+\gamma}}\le D$ for any $y\in \Bbb R^{d+2}$ and some $D>0,\gamma>0$, then given $p>d$, for every covariant  tensor ${\bf T}$ there holds
\begin{align*}
\max_{\Sigma}|{\bf T}|_{g}\le C\left((\int_{\Sigma}|\nabla {\bf T}|^p_gd\mu)^{\frac{1}{p}}+(\int_{\Sigma}| {\bf T}|^p_gd\mu)^{\frac{1}{p}}\right),
\end{align*}
where $C$ depends only on $d,p,\gamma,D$.
\end{lemma}
\begin{proof}
This lemma is essentially due to [Page 157-159, Mantegazza \cite{Maz}]. We point out that the original result of  \cite{Maz} was stated for compact manifolds $(\mathcal{M},h)$ with finite volume.
But carefully checking \cite{Maz}'s proof, we see her argument indeed only requires the uniform bounds of
$\mu(B_1(y)\bigcap \Sigma)$  for any $y\in \Bbb R^{d+2}$ and $\|{\bf H}\|_{L^{d+\gamma}}$.
\end{proof}

\subsection{Local Cauchy problem }

Let's move to local well-posedness.
Our main result for this section is as follows. 
\begin{theorem}\label{2K}
\begin{itemize}
  \item Given $d\ge 2$, let $k_0$ be the smallest integer such that $k_0>\frac{d}{2}+2$, i.e. $k_0=[\frac{d}{2}]+3$. Assume that $u^0_1, u^0_2:\Bbb R^d\to \Bbb R$ are
functions belonging to  $H^{\sigma}$ for any $\sigma\ge 0$.  There exists a small constant $T>0$ depending only on
$$\|u^0_1\|_{H^{k_0}}+\|u^0_2\|_{H^{k_0}},$$
such that  there exists a unique smooth local solution $F$ to SMCF  of the form
\begin{align}\label{tFgm}
F(t,x)=(x_1,...,x_d,u_1(t,x),u_2(t,x)),\mbox{  }(u_1,u_2)\upharpoonright_{t=0}=(u^0_1,u^0_2),
\end{align}
so that $u_1,u_2\in C([-T,T];H^{\sigma})$ for any $\sigma\ge 0$.
  \item Given $d\ge 2$, let $s>\frac{d+5}{2}$. Then there exists $\epsilon>0$  sufficiently small
such that for any initial data $u_0=(u^0_1,u^0_2)$  with
$$\|u_0\|_{H^{s}}\le \epsilon, $$
the SMCF equation for $F$ of the form (\ref{tFgm}) is locally well-posed in $H^s(\Bbb R^d)$ on the time interval $[-1,1]$.
\end{itemize}
\end{theorem}

We begin with the small data case. 
\begin{lemma}\label{HjK}
 Given $d\ge 2$, let $s>\frac{d+5}{2}$. Then there exists $\epsilon>0$  sufficiently small
such that for any initial data $u_0=(u^0_1,u^0_2)$  with
$$\|u_0\|_{H^{s}}\le \epsilon, $$
the SMCF equation for $F$ of the form (\ref{tFgm}) is locally well-posed in $H^s(\Bbb R^d)$ on the time interval $[-1,1]$.
\end{lemma}
\begin{proof}
This follows directly by [Theorem 1.1, \cite{MMT1}] and (\ref{H3M}).
\end{proof}

Now, let us prove the local well-posedness theory for arbitrary large data stated in Theorem \ref{2K}. The whole proof is divided into two steps. In Step 1, we prove the existence of solution to SMCF in $t\in[-T,T]$ with $T>0$ depending on
$$\|u^0_1\|_{H^{K}}+\|u^0_2\|_{H^{K}},$$
where $K$ is some integer depending only on $d$ and larger than $k_0$.
In Step 2, we prove the existence of solution to SMCF in $t\in[-T,T]$ with $T>0$ depending only on
$$\|u^0_1\|_{H^{k_0}}+\|u^0_2\|_{H^{k_0}}.$$
\begin{lemma}\label{FFF}
Given $d\ge 2$, let $k_0$ be the smallest integer such that $k_0>\frac{d}{2}+2$, i.e. $k_0=[\frac{d}{2}]+3$.  Let  $(u^0_1,u^0_2)\in H^{\infty}$. Then there exists a  constant $T>0$ depending only on
$$\|u^0_1\|_{H^{k_0}}+\|u^0_2\|_{H^{k_0}},$$
such that  there exists a unique smooth local solution $F$ to SMCF  of the form
\begin{align*}
F(t,x)=(x_1,...,x_d,u_1(t,x),u_2(t,x)),\mbox{  }(u_1,u_2)\upharpoonright_{t=0}=(u^0_1,u^0_2),
\end{align*}
so that $u_1,u_2\in C([-T,T];H^{\sigma})$ for any $\sigma>0$.
\end{lemma}
\begin{proof}
Let $\lambda\in (0,1)$. Consider the flow
\begin{align}\label{perturb}
\partial_tF^{\lambda}=J(F^{\lambda}){\bf H}(F^{\lambda})+\lambda {\bf H}(F^{\lambda}), \mbox{ }F^{\lambda}(0,x)=(x_1,...,x_d,u^0_1(x),u^0_2(x)).
\end{align}
The local well-posedness of (\ref{perturb}) in $H^s$ with $s>\frac{d}{2}+2$ follows by [Prop. 8.1, \cite{Tay}].  The strong parabolic property of  (\ref{perturb}) can be easily seen from
\begin{align*}
\left\{ \begin{gathered}
  {\partial _t}u_1^\lambda  = \lambda (\frac{{{{\left| {{\partial _x}u_1^\lambda \cdot{\partial _x}u_2^\lambda } \right|}^2}}}
{{{\Lambda ^2}{{(1 + |{\partial _x}u_1^\lambda {|^2})}^2}}} + \frac{1}
{{1 + |{\partial _x}u_1^\lambda {|^2}}}){g^{ij}}\partial _{{x_i}{x_j}}^2u_1^\lambda {\text{ }} - \lambda \frac{{{\partial _x}u_1^\lambda \cdot{\partial _x}u_2^\lambda }}
{{{\Lambda ^2}(1 + |{\partial _x}u_1^\lambda {|^2})}}{g^{ij}}\partial _{{x_i}{x_j}}^2u_2^\lambda  \hfill \\
  \mbox{ }\mbox{ }\mbox{ }\mbox{ } - \frac{1}
{{\Lambda \sqrt {1 + |{\partial _x}u_1^\lambda {|^2}} }}{g^{ij}}\partial _{{x_i}{x_j}}^2u_2^\lambda  \hfill \\
  {\partial _t}u_2^\lambda  = \lambda \frac{1}
{{{\Lambda ^2}}}{g^{ij}}\partial _{{x_i}{x_j}}^2u_2^\lambda {\text{ }} - \lambda \frac{{{\partial _x}u_1^\lambda \cdot{\partial _x}u_2^\lambda }}
{{{\Lambda ^2}(1 + |{\partial _x}u_1^\lambda {|^2})}}{g^{ij}}\partial _{{x_i}{x_j}}^2u_1^\lambda  + \frac{1}
{{\Lambda \sqrt {1 + |{\partial _x}u_1^\lambda {|^2}} }}{g^{ij}}\partial _{{x_i}{x_j}}^2u_1^\lambda.  \hfill \\
\end{gathered}  \right.
\end{align*}
We also recall that the solution of  (\ref{perturb}) belongs to $C([0,T'];H^s)$ as long as
\begin{align}\label{JM1}
\|u^{\lambda}\|_{L^{\infty}([0,T'];C^{1+\delta_1})}<\infty \mbox{ }{\rm for} \mbox{ }{\rm some}\mbox{ }  \delta_1>0.
\end{align}
Direct calculations show (\ref{e4}) also holds for $\mathbf{A}^{\lambda}$, the second fundamental form associated with $F^{\lambda}$.
Define $T_{\lambda}\in [0,1]$ to be the maximal time such that
\begin{align}\label{Fvb}
\|u^{\lambda}\|_{L^{\infty}([0,T_{\lambda}];H^{k_0})}\le C_0\varepsilon\ll 1.
\end{align}
 Let $k_1=[\frac{d}{2}]+4$.
Define $T_{\lambda}\in [0,1]$ to be the maximal time such that
\begin{align}\label{nnFvb}
\|u^{\lambda}\|_{C([0,T_{\lambda}];H^{k_1})}\le C_0\|u_0\|_{H^{K}},
\end{align}
where $K\ge k_1$ is an integer depending only on $d$ to be determined later, and $C_0>1$ is a constant to be chosen later. By the blow-up criterion and Sobolev embedding, $u^{\lambda}$ is smooth for $t\in[0,T_{\lambda}]$.
We shall use a continuity argument to show $T_{\lambda}\ge T$ for some $T>0$ depending only on
$\|u_0\|_{H^{K}}$.  shows  $u^{\lambda}$ is smooth for $t\in[0,T_{\lambda}]$.

First,  letting $\phi^{\lambda}=u^{\lambda}_1+ \mathrm{{i}}u^{\lambda}_2$, by
Duhamel principle, we infer from   (\ref{nnFvb}) and the fact  $H^{k_1-2}$ is an algebra   that
\begin{align}\label{nnFvbzz}
\|u^{\lambda}\|_{C([0,T_{\lambda}];H^{k_1-2})}\le \|u_0\|_{H^{k_1}_x}+C_1T_{\lambda}[C_0\|u_0\|_{H^{K}_x}+ C^{n_0}_0\|u_0\|^{n_0}_{H^{K}_x}],
\end{align}
where $n_0\in\Bbb Z_+$ and  $C_1>0$ depend only on $d$, see Remark \ref{12}.
Second, note that  (\ref{e1})  with (\ref{nnFvb})  implies
\begin{align}\label{IKJ}
\|{\bf A}^{\lambda}\|_{C([0,T_{\lambda}];H^{0,\infty})}\le C_2  C_0\|u_0\|_{H^{K}},
\end{align}
for some universal constant $C_2>0$.
Then for any $k\in\Bbb Z_+$, (\ref{e4}), (\ref{e2}), and Gronwall inequality give
\begin{align}\label{nnFvbzpz}
\|{\bf A}^{\lambda}\|_{C([0,T_{\lambda}];H^{k,2})}\le (C_3\|u_0\|_{H^{k}}+C_{3}\|u_0\|^{n_1}_{H^{k}})e^{C_2T_{\lambda}C^2_0\|u_0\|^2_{H^{K}}},
\end{align}
where $n_1\in\Bbb Z_+, C_3>0$ are independent of $C_0$, and depend  only on $d,k$. Hence, Lemma \ref{hbnjkmkk}, (\ref{nnFvbzpz}) and (\ref{nnFvb}) yield that there exist some sufficiently large $K\in\Bbb Z_+$  depending only on $d$ and a polynomial $P_{K}$ such that
\begin{align}
&\|D^2u^{\lambda}\|_{C([0,T_{\lambda}];H^{k_1-2})}\le   G(\|Du^{\lambda}\|_{C([0,T_{\lambda}];L^{\infty}_x)}) P_{K}(\|{\bf A}^{\lambda}\|_{C([0,T_{\lambda}];H^{K,2})})\nonumber\\
 &\le G(\|Du^{\lambda}\|_{C([0,T_{\lambda}];L^{\infty}_x)})(C_3\|u_0\|_{H^{K}}+C_{3}\|u_0\|^{n_2}_{H^{K}})e^{C_2T_{\lambda}C^2_0\|u_0\|^2_{H^{K}}},\label{nnFvbzzz}
\end{align}
where $n_2\in\Bbb Z_+$ and $C_3>0$ depend only on $k_0,d$, and $G:\Bbb R^+\to \Bbb R^{+}$ is of the form
\begin{align*}
G(y)=C_4[(1+y^2)^{\frac{k_1+3}{2}}+Q_{k_1-2}(y)],
\end{align*}
for some constant  $C_4>0$ depending only on $d$. Let $n_3$ denote the top order of $Q_{k_0-2}$.
Since $k_0>\frac{d}{2}+3$, (\ref{nnFvbzz}) and Sobolev embedding show
\begin{align*}
\|Du^{\lambda}\|_{C([0,T_{\lambda}];L^{\infty}_x)}\le C_5\|u_0\|_{H^{K}}+C_{5}T_{\lambda}C^{n_0}_0\|u_0\|^{n_0}_{H^{K}}
\end{align*}
for some constant $C_5>0$ depending only on $d$.
Therefore, (\ref{nnFvbzzz}) now leads to
\begin{align}\label{nnnnFvbzzz}
\|D^2u^{\lambda}\|_{C([0,T_{\lambda}];H^{k_1-2})}\le e^{C_2T_{\lambda}C^2_0\|u_0\|^2_{H^{K}}}  C_6[\|u_0\|_{H^{K}}+\|u_0\|^{n_*}_{H^{K}}+(T_{\lambda}C_{0}\|u_0\|_{H^{K}})^{n_*}]
\end{align}
where we denote $n_*=n_0+n_1+n_2+n_3+k_1+3$, and $C_6>0$ depends only on  $C_1,...,C_5$ and thus depends only on $k_0,d$.

Thus set  $C_0>1$ to satisfy
\begin{align*}
C_0\ge 100 C_6+100\|u_0\|^{n_*-1}_{H^{K}}+100,
\end{align*}
and take $T>0$ to be sufficiently small such that
\begin{align}\label{HddjnM3}
e^{C_2T(C^2_0\|u_0\|^2_{H^{K}})}&\le 2; \mbox{  }T C_{0}\le \frac{1}{2}.
\end{align}
Then by (\ref{nnnnFvbzzz}), (\ref{nnFvbzz}), we get $T_{\lambda}\ge T$ for all $\lambda\in (0,1)$. So there holds
\begin{align}\label{A.T1}
\|u^{\lambda}\|_{C([0,T];H^{k_1})}\lesssim 1.
\end{align}
Therefore, (\ref{IKJ}) shows
\begin{align*}
\|{\bf A}^{\lambda}\|_{C([0,T];H^{0,\infty})}\lesssim 1.
\end{align*}
Hence,  (\ref{e4}) implies for given $l\ge 0$ there exists $C_l>0$ such that
\begin{align}\label{A.T2}
\|{\bf A}^{\lambda}\|_{C([0,T];H^{l,2})}\le   C_{l}.
\end{align}

{\bf Step 3.}  By Lemma \ref{hbnjkmkk}, (\ref{A.T1}), (\ref{A.T2}), $u^{\lambda}\in C([0,T];H^{k})$ for any given  $k\ge 0$ with uniform bounds independent of $\lambda\in(0,1)$.
Thus, there exists a  sequence $\lambda_n\to 0$ such that $u^{\lambda_{n}}$ converges to $u\in  H^{\infty}$  which gives rise to a smooth solution $F$ of SMCF of the form (\ref{tFgm}) in $t\in[0,T]$ with $0<T\ll1 $ fulfilling (\ref{HddjnM3}).
A time reflection and defining $J(F)$ to be the  opposite direction  rotation can cover  $t\in[-T,0]$ as well.

{\bf Step 4.} The  uniqueness follows by \cite{Song2}. We remark that  although \cite{Song2} stated uniqueness for  compact $\Sigma$, its proof indeed does not  use compactness of $\Sigma$. In fact,  in \cite{Song2}  the compactness of $\Sigma$ was   only used in the existence part.
\end{proof}

\begin{remark}
 Lemma \ref{FFF} indeed yields an existence theory of strong solutions. In fact,  Lemma \ref{FFF} implies there exists an integer $K$ depending only on $d$ such that given an initial data $u_0\in H^K$, there exists $T>0$ depending only on $\|u_0\|_{H^{K}}$ so that SMCF has a strong solution $u\in L^{\infty}([-T,T];H^{k_1})$ with initial data $u_0$. This can be proved by density arguments and the uniform bounds of $\|u^{\lambda}\|_{L^{\infty}_tH^{k_1}_x}$ obtained in Lemma \ref{FFF}. In the case $d=2$, given an integer $l\ge 4$,  we can get a strong solution $u\in C([-T,T];H^{l}(\Bbb R^2))$ of SMCF with initial data $u_0\in H^{l}(\Bbb R^2)$. This improvement is due to the stronger result (\ref{e3p}) in  $d=2$.
\end{remark}

To improve Lemma \ref{FFF}, we need the following blow-up criterion.
\begin{corollary}\label{GBnM}
Let  $F$ be a local smooth solution  to SMCF   of the form
\begin{align*}
F(t,x)=(x_1,...,x_d,u_1(t,x),u_2(t,x)),\mbox{  }(u_1,u_2)\upharpoonright_{t=0}=(u^0_1,u^0_2),
\end{align*}
such that $u_1,u_2\in C([0,T_*);H^{\infty})$. Then as long as
\begin{align*}
\|{\bf A}\|_{L^{\infty}_t([0,T_*); H^{0,\infty})}<\infty,
\end{align*}
 $F$ can be continuously extended to a smooth solution in $[0, T_*+\rho]$ for some $\rho>0$.
\end{corollary}
\begin{proof}
Assume that for some $C_1>0$,
\begin{align*}
\|{\bf A}\|_{L^{\infty}_t([0,T_*); H^{0,\infty})}\le C_1.
\end{align*}
By (\ref{Mazw}), we get
\begin{align}\label{cfgvbjkimnm}
\partial_t|Du|^2\le 2C^2_1(|Du|^2+1).
\end{align}
So Gronwall inequality shows
\begin{align}\label{dxxfghj}
\|Du\|_{L^{\infty}_t([0,T_*); L^{\infty}_x)}\lesssim_{C_1,T_*} \|Du_0\|_{L^{\infty}_x}+1.
\end{align}
Meanwhile, given $l\in \Bbb N$, (\ref{e4}) with Gronwall inequality implies there exists $B_l>0$ such that
\begin{align}\label{1dxxfghj}
\|{\bf A}\|_{L^{\infty}_t([0,T_*); H^{l,2})}\lesssim_{T_*,C_1}B_{l}.
\end{align}
By  Lemma \ref{hbnjkmkk}, (\ref{dxxfghj}) and (\ref{1dxxfghj}) imply for any given $l\in\Bbb N$
\begin{align*}
\|D^2u\|_{L^{\infty}_t([0,T_*); H^{l}_x)}\lesssim_{l,T_*,C_1} 1.
\end{align*}
Integrating (\ref{cfgvbjkimnm}) in $\Bbb R^d$, we deduce from Gronwall inequality that
\begin{align*}
\|Du\|_{L^{\infty}_t([0,T_*);L^{2}_x)}\lesssim_{C_1,T_*}  \|Du_0\|_{L^{2}_x}+1.
\end{align*}
And one has
\begin{align*}
\frac{d}{dt}\|u\|^2_{ L^{2}_x}\le \| {\bf H}\|_{H^{0,2}}\|u\|_{L^2_x}.
\end{align*}
Thus  Gronwall inequality  yields
\begin{align*}
\|u\|_{L^{\infty}_t([0,T_*);L^{2}_x)}\lesssim_{C_1,T_*}  \|u_0\|_{L^{2}_x}+1.
\end{align*}
In a summary, we get
$$\|u\|_{L^{\infty}_t([0,T_*);H^{k}_x)}<\infty, \mbox{ }\forall k\in \Bbb Z_+.$$
By Lemma \ref{FFF}, $u$ can be continuously extended beyond the interval $0\le t<T_*$.
\end{proof}

\begin{lemma}\label{11HjK}
Given $d\ge 2$, let $k_0$ be the smallest integer such that $k_0>\frac{d}{2}+2$, i.e. $k_0=[\frac{d}{2}]+3$.  Let  $(u^0_1,u^0_2)\in H^{\infty}$. Then there exists a  constant $T>0$ depending only on
$$\|u^0_1\|_{H^{k_0}}+\|u^0_2\|_{H^{k_0}},$$
such that  there exists a unique smooth local solution $F$ to SMCF  of the form
\begin{align*}
F(t,x)=(x_1,...,x_d,u_1(t,x),u_2(t,x)),\mbox{  }(u_1,u_2)\upharpoonright_{t=0}=(u^0_1,u^0_2),
\end{align*}
so that $u_1,u_2\in C([-T,T];H^{\sigma})$ for any $\sigma>0$.
\end{lemma}
\begin{proof}
It suffices to prove the lifespan in Lemma \ref{FFF} in fact can be improved to depend only on $\|u_0\|_{H^{k_0}}$.
Given $u_0\in H^{\infty}$, by Lemma \ref{FFF} there is a smooth graph solution to SMCF in $t\in(-T_1,T_1)$ for some $T_1>0$. Let $T_*>0$ be the maximal time such that
\begin{align}\label{GGGbb}
\|{\bf A}\|_{L^{\infty}_{t}((-T_*,T_*);H^{0,\infty})}\le C_*\|u_0\|_{H^{k_0}},
\end{align}
where $C_*>1$ is a constant to be determined later. It is easy to see $T_*>0$ by the continuity of $u(t)$ in $t$ and Sobolev embeddings. Moreover, we see $T_*\le T_1$ by the blow-up criterion in Corollary \ref{GBnM}. In the rest   we prove $T_*\ge T$ for some $T>0$ depending only on $\|u_0\|_{H^{k_0}}$.

By (\ref{GGGbb}), (\ref{e2}), (\ref{e4}) and Gronwall inequality, we get
\begin{align}\label{V31}
\sup_{t\in (-T_*,T_*)}\|{\bf A}\|_{H^{[\frac{d}{2}]+1,2}}\le C_1e^{CT_*C^2_*\|u_0\|^2_{H^{k_0}}}(\|u_0\|_{H^{k_0}}+\|u_0\|^{m}_{H^{k_0}}),
\end{align}
where $m\in\Bbb Z_+$, $C>0$, $C_1>0$ are universal constants depending only on $d$.
At $t=0$, by Sobolev embedding and (\ref{e1}), one has
\begin{align}\label{V01}
 \|{\bf A}_0\|_{H^{0,2[\frac{d}{2}]+2}}\lesssim \|u_0\|_{H^{k_0}}.
\end{align}
And applying Hamilton's interpolation inequality (see Lemma \ref{Hami}) to ${\bf A}$  shows
\begin{align}\label{V5}
\|\nabla {\bf A}\|_{H^{0,2[\frac{d}{2}]+2}}\le C_2\|{\bf A}\|^{1-\frac{1}{[\frac{d}{2}]+1}}_{H^{0,\infty}}\|{\bf A}\|^{\frac{1}{[\frac{d}{2}]+1}}_{H^{[\frac{d}{2}]+1,2}},
\end{align}
where $C_2>0$ depends only on $d$.
Then (\ref{Maz1}) and Gronwall inequality show
\begin{align*}
&\sup_{t\in (-T_*,T_*)}\|{\bf A}\|^{2[\frac{d}{2}]+2}_{H^{0,2[\frac{d}{2}]+2}}\\
&\lesssim e^{CT_*+CT_*C^2_*\|u_0\|^2_{H^{k_0}}}\left(\|u_0\|^{2[\frac{d}{2}]+2}_{H^{k_0}}+T_*(C_*\|u_0\|_{H^{k_0}})^{2[\frac{d}{2}]}\|{\bf A}\|^2_{L^{\infty}((-T_*,T_*);H^{[\frac{d}{2}]+1,2})}\right),
\end{align*}
which, together with (\ref{V31}), further gives
\begin{align}\label{Gg56}
 \sup_{t\in (-T_*,T_*)}\|{\bf A}\|_{H^{0,2[\frac{d}{2}]+2}}\le   \Phi\left(T_*(C_*\|u_0\|_{H^{k_0}})^2+T_*(C_*\|u_0\|_{H^{k_0}})^{2[\frac{d}{2}]},\|u_0\|_{H^k_0}\right),
\end{align}
for some universal function $\Phi:\Bbb R^+\times \Bbb R^+\to \Bbb R^+$ which is increasing in both variables.
For simplicity of notations, we define
$$\Omega:=T_*(C_*\|u_0\|_{H^{k_0}})^2+T_*(C_*\|u_0\|_{H^{k_0}})^{2[\frac{d}{2}]}.$$
Then, (\ref{Gg56}) implies
\begin{align}\label{ws}
\|{\bf H}\|_{L^{\infty}_t((-T_*,T_*);H^{0,2[\frac{d}{2}]+2})}\le  \Phi(\Omega,\|u_0\|_{H^{k_0}}).
\end{align}
Let $B_1(y)$ be a unit ball in $\Bbb R^{n+2}$. Since $\partial_t\mu(t)=0$ along the SMCF, we see for all $t\in (-T_*,T_*)$,
\begin{align}\label{V2}
\mu_t(B_1(y)\bigcap \Sigma_t)\le \mu_0(\{x\in\Bbb R^d: |x-y'|\le 1\})\le C(\|Du_0\|_{L^{\infty}_x})\le C(\|u_0\|_{H^{k_0}_x}),
\end{align}
where $y'\in\Bbb R^d$ is the projection of $y\in\Bbb R^{d+2}$ onto $\Bbb R^d$.
By (\ref{ws}) and (\ref{V2}),  we have  verified the conditions in Lemma \ref{Maz}. So, given $p>d$, Lemma \ref{Maz} implies
\begin{align}\label{V4}
\|{\bf A}\|_{L^{\infty}([-T_*,T_*]; H^{0,\infty})}\le C'_1(\|{\bf A}\|_{L^{\infty}([-T_*,T_*]; H^{0,p})}+\|\nabla {\bf A}\|_{L^{\infty}([-T_*,T_*]; H^{0,p})}),
\end{align}
for some $C'_1>0$ depending only on $p,d$ and
$\Omega$.
Taking $p=2[\frac{d}{2}]+2$, we thus obtain by (\ref{V4}) and (\ref{V5}) that
\begin{align}
 \|{\bf A}\|_{L^{\infty}([-T_*,T_*]; H^{0,\infty})}&\lesssim \|{\bf A}\|_{L^{\infty}([-T_*,T_*]; H^{0,2[\frac{d}{2}]+2})}+\|{\bf A}\|_{L^{\infty}([-T_*,T_*]; H^{[\frac{d}{2}]+1,2})},\label{V6}
\end{align}
where the implicit constant  depends only on $\Omega$ and $d$.
Using H\"older inequality
\begin{align*}
\|{\bf A}\|_{H^{0,2[\frac{d}{2}]+2}}&\le  \|{\bf A}\|^{\frac{1}{[\frac{d}{2}]+1}}_{H^{0,2}}  \|{\bf A}\|^{1-\frac{1}{[\frac{d}{2}]+1}}_{H^{0,\infty}},
\end{align*}
(\ref{V6}) further yields
\begin{align}\label{HgjL}
\|{\bf A}\|_{L^{\infty}([-T_*,T_*];H^{0,\infty})}&\lesssim \|{\bf A}\|_{L^{\infty}([-T_*,T_*];H^{[\frac{d}{2}]+1,2})},
\end{align}
where the implicit constant  depends only on $\Omega,d$.
Hence, by (\ref{V31}) and (\ref{HgjL}) we obtain
\begin{align}\label{V67}
\|{\bf A}\|_{L^{\infty}([-T_*,T_*];H^{0,\infty})}\le \Phi_1(\Omega,\|u_0\|_{H^{k_0}})
\end{align}
for some universal  function $\Phi_1:\Bbb R^+\times\Bbb R^+\to \Bbb R^+$ which is increasing in both variables.

Now, take $C_*$ to be sufficiently large such that
\begin{align}\label{1afghjk}
C_*\|u_0\|_{H^{k_0}}\ge 2\Phi_1( 1,\|u_0\|_{H^{k_0}})+2,
\end{align}
and choose $T>0$ to be sufficiently small such that
\begin{align}\label{2afghjk}
 \Omega=T_*(C_*\|u_0\|_{H^{k_0}})^2+T_*(C_*\|u_0\|_{H^{k_0}})^{2[\frac{d}{2}]}\le \frac{1}{2}.
\end{align}
Then (\ref{V67})  reveals
\begin{align*}
\|{\bf A}\|_{L^{\infty}_t([-T,T];H^{0,\infty})}\le \frac{1}{2}C_*\|u_0\|_{H^{k_0}}.
\end{align*}
So one has $T_*\ge T$ by (\ref{GGGbb}). It is obvious that $T>0$ defined via (\ref{1afghjk})-(\ref{2afghjk}) depends only on $\|u_0\|_{H^{k_0}}$.
\end{proof}

 \section{Proof of Theorem 1.1 in $d\ge 2$}

By Duhamel principle,
the solution of (\ref{mss}) can be expressed by

\begin{align}\label{1ma}
 \phi(t)=e^{\mathrm{i}t\Delta}\phi_0+\int^{t}_{0}e^{\mathrm{i}(t-s)\Delta} O({\partial^2_x\phi }|\partial_x\phi|^2)(s)ds.
\end{align}

Let $k$ be the smallest integer such that $k>\max(\frac{d}{2}+3.5,d+1)$. Let $1<q<2$, $q'=\frac{q}{q-1}$, and
\begin{align}\label{huashan}
\frac{1}{q}=\frac{1}{d}+\frac{1}{2}-\delta,
\end{align}
where $\delta>0$ is sufficiently small.

Let $\mathcal{T}$ be the maximal time such that
\begin{align}\label{1mab1}
\sup_{t\in[-\mathcal{T},\mathcal{T}]}\left(\langle t\rangle^{\frac{d}{2}(\frac{2}{q}-1)}\| \phi(t)\|_{W^{2,q'}}+\| \phi(t)\|_{H^{k}}\right)\le C_0\epsilon,
\end{align}
where $C_0>1$ is to be determined later.
Let's prove Theorem 1.1 by bootstrap.

{\bf Step 1.} By local theorem in Lemma \ref{HjK} and Sobolev embedding, one has  $\mathcal{T}\ge 1$.

{\bf Step 2.}
Assume that $1<|t|< \mathcal{T}$, and without loss of generality let $t\in[1,\mathcal{T}]$.  Then  (\ref{1ma}) and linear dispersive estimates of $e^{\mathrm{i}\Delta t}$ (see Lemma 6.2) give
\begin{align}\label{1ma1}
\| \phi(t)\|_{W^{2,q'}}\lesssim t^{\frac{d}{2}(1-\frac{2}{q})}\|\phi_0\|_{W^{2,q}}+\int^{t}_{0}\sum^{2}_{l=0}(t-s)^{\frac{d}{2}(1-\frac{2}{q})}\|\partial^{l}O({\partial^2_x\phi }|\partial_x\phi|^2)(s)\|_{L^q_x}ds.
\end{align}
{\it 1.  Leading cubic terms.}
The leading cubic part of RHS of (\ref{1ma1}) is
\begin{align}
&\int^{t}_{0}(t-s)^{\frac{d}{2}(1-\frac{2}{q})} \||{\partial^3_x\phi }||\partial^2_x \phi|\partial_x\phi|\|_{L^q_x}ds +\int^{t}_{0}(t-s)^{\frac{d}{2}(1-\frac{2}{q})} \||{\partial^4_x\phi }| |\partial_x\phi|^2\|_{L^{q}_x}ds\label{Fghlknm}\\
&+{{\rm low\mbox{ } derivative\mbox{ } terms}}.
\end{align}
Gagliardo-Nirenberg inequality shows
\begin{align*}
& \|{\partial_x\phi }\|_{L^{\frac{2q}{2-q}}_x}\lesssim  \|\phi\|^{\omega}_{L^{q'}_x} \|{\partial^k_x\phi }\|^{1-\omega }_{L^{2}_x}, \mbox{ }
\omega=\frac{k-1-\frac{d}{q'}}{k+d(\frac{1}{2}-\frac{1}{q})},\\
 &\|{\partial^2_x\phi }\|_{L^{\frac{2q}{2-q}}_x}\lesssim  \|{\partial_x\phi }\|^{\theta}_{L^{q'}_x} \|{\partial^k_x\phi }\|^{1-\theta }_{L^{2}_x},
\mbox{ }    \theta=\frac{k-2-\frac{d}{q'}}{k-1+d(\frac{1}{2}-\frac{1}{q})}.
\end{align*}
When $k\ge d+2$, one has for $q$ in (\ref{huashan}) with $0<\delta\ll 1$,
\begin{align}
 \frac{d}{2}(1-\frac{2}{q})\min(2\theta,2\omega)&<-1\label{B91}\\
\min( \theta, \omega)&> \frac{1}{2}.\label{B92}
\end{align}
Then (\ref{1mab1})  shows  (\ref{Fghlknm}) is dominated by
\begin{align*}
&\int^{t}_{1} (t-s)^{\frac{d}{2}(1-\frac{2}{q})}\|\partial^4_x\phi\|_{L^2_x}\|{\partial^2_x\phi }\|_{L^{\frac{2q}{2-q}}_x}\|\partial_x\phi|\|_{L^{\frac{2q}{2-q}}_x}ds
+ \int^{1}_{0} (t-s)^{\frac{d}{2}(1-\frac{2}{q})}\|{\partial^4_x\phi }\|_{L^2_x}\|\partial_x\phi\|^2_{L^{\frac{2q}{2-q}}_x}ds\\
&\lesssim \epsilon^3\int^{t}_{1}(t-s)^{\frac{d}{2}(1-\frac{2}{q})}s^{\frac{d}{2}(1-\frac{2}{q})2\min(\theta,\omega)}ds+ \int^{1}_0  (t-s)^{\frac{d}{2}(1-\frac{2}{q})}\|\phi \|^3_{H^{k}} ds\\
&\lesssim \langle t\rangle^{\frac{d}{2}(1-\frac{2}{q})}\epsilon^3,
\end{align*}
provided that $t\in[-\mathcal{T},\mathcal{T}]$, where we applied Sobolev embedding for $s\in [0,1]$ due to the assumption $k>\frac{d}{2}+3.5$.
The low derivatives terms are easier to handle and also contribute to (\ref{1ma1}) by  $\langle t\rangle^{\frac{d}{2}(1-\frac{2}{q})}\epsilon^3$.

{\it 2. Remained higher order terms.} The higher order remainders are easy to dominate, since we always have by (\ref{mq}) that
\begin{align*}
\int^{t}_{0} (t-s)^{\frac{d}{2}(1-\frac{2}{q})}\|\mathcal{R}\|_{W^{2,q}_x}ds&\lesssim
\int^{t}_{0}(t-s)^{\frac{d}{2}(1-\frac{2}{q})} \|\phi\|_{W^{4,2}_x}\|\phi\|^2_{W^{2,\frac{2q}{2-q}}_x}\|  \phi\|_{W^{1,\infty}_x}ds
\end{align*}
which is  also admissible by (\ref{B91}) and (\ref{B92}).

 In a summary we have proved in Step 2 that
 \begin{align}\label{musu1}
\| \phi(t)\|_{W^{2,q'}}\le C_1\langle  t\rangle^{\frac{d}{2}(1-\frac{2}{q})}[ \|\phi_0\|_{W^{2,q}}+C^3_0\epsilon^3]
\end{align}
 for some $C_1>0$ depending only on $d$.

{\bf Step 3.}  Let us deal with the $H^{k}$ norm in (\ref{1mab1}).
By (\ref{musu1}) and Gagliardo-Nirenberg inequality, we see
 \begin{align}\label{VbbB}
\|D^2 \phi(t)\|_{L^{\infty}_x}\lesssim \|\phi\|^{\theta_1}_{W^{2,q'}_x}\|\phi\|^{1-\theta_1}_{H^k_x} \lesssim \langle  t\rangle^{\theta_1\frac{d}{2}(1-\frac{2}{q})} \epsilon,
\end{align}
where $\theta_1\in (0,1)$ is given by
 \begin{align*}
 \theta_1=\frac{k-2-\frac{d}{2}}{d(\frac{1}{2}-\frac{1}{q})+k-2}.
\end{align*}
It is easy to verify
 \begin{align*}
 \theta_1 d(\frac{2}{q}-1)>1,
\end{align*}
for $k\ge d+2$ and $0<\delta\ll 1$.
Hence, for any $t\in[0,\mathcal{T}]$  we get
\begin{align*}
\int^{t}_0\|D^2 \phi(t)\|^2_{L^{\infty}_x}ds\lesssim  \epsilon^2.
\end{align*}
Then
by (\ref{e1}), for $t\in[0,\mathcal{T}]$ we have
\begin{align*}
\int^{t}_0\max_{\Sigma}|{\bf A}|^2_g(s)ds\lesssim \epsilon^2.
\end{align*}
Thus (\ref{e4}) and Gronwall inequality show for any $t\in[0,\mathcal{T}]$
\begin{align*}
\int_{\Sigma} |\nabla^{l}{\bf A}|^2_g(t)d\mu\le 2\int_{\Sigma} |\nabla^{l}{\bf A}|^2_g(0)d\mu,
\end{align*}
if $0<\epsilon\ll1$. Then (\ref{e2}) shows for any $t\in[0,\mathcal{T}]$, $0\le l\le k-2$,
\begin{align}\label{MkL}
\int_{\Sigma} |\nabla^{l}{\bf A}|^2_g(t)d\mu\lesssim \|\phi_0\|_{H^{k}}\ll 1,
\end{align}
provided that  $0<\epsilon\ll 1$.
Since the RHS of (\ref{e3}) with $l=k-2$  is quadratic in $\|D^2u\|_{W^{k-2,2}}$ and  $\|D^2u\|_{W^{k-2,2}}$ is small by bootstrap assumption,  (\ref{MkL}) further implies for any $t\in [0,\mathcal{T}]$,
\begin{align}\label{HnbgM}
\|D^2u(t)\|_{W^{k-2,2}}\le C\|\phi_0\|_{H^{k}}
\end{align}
for some $C>0$ if $0<\epsilon\ll1$.  (\ref{HnbgM}) provides admissible  bounds of $\|\phi\|_{{\dot H}^{l}}$ with $2\le l\le k$. Using (\ref{1ma}) we get
\begin{align}\label{tmamm}
\| \phi(t)\|_{H^2_x}\lesssim  \|\phi_0\|_{H^2_x}+ \int^{t}_{0}\sum^{2}_{l=0}\|\partial^{l}O({\partial^2_x\phi }|\partial_x\phi|^2)(s)\|_{L^2_x}ds.
\end{align}
We observe  that the RHS of (\ref{tmamm}) is dominated by $\|\phi_0\|_{H^2_x}+ C^3_0\epsilon^3$ up to a universal constant depending only on $d$ via noting  that (\ref{VbbB}) yields
\begin{align*}
\|\partial^2_x\phi\|_{L^{\infty}_x}+\|\partial_x\phi\|_{L^{\infty}_x} \lesssim \epsilon \langle t\rangle^{-\alpha},
\end{align*}
for some $\alpha>\frac{1}{2}$.

As a summary, we have obtained  admissible  bounds of $\|\phi\|_{H^{k}}$:
\begin{align}\label{HnbM}
\|u(t)\|_{H^{k}}\le C_2[ \|\phi_0\|_{H^{k}}+C^3_0\epsilon^3]
\end{align}
for some $C_2>0$ depending only on $d$ and any $t\in [0,\mathcal{T}]$. The inverse direction $t\in[-\mathcal{T},0]$ follows by a time reflection and defining $J(F)$ as the opposite direction rotation. Hence, (\ref{HnbM}) holds for all $t\in[-\mathcal{T},\mathcal{T}]$.

{\bf Step 4.}
Set $C_0$ to satisfy
\begin{align*}
C_0>4+4C_1+4C_2,
\end{align*}
then choose $\epsilon>0$ to be sufficiently small such that
\begin{align*}
C_2C^2_0\epsilon^2\le \frac{1}{4}.
\end{align*}
By Step 1, (\ref{musu1}) in Step 2, (\ref{HnbM}) in Step 3 and bootstrap, we have proved (\ref{1mab1}) holds with $C_0\epsilon$ replaced by $\frac{1}{2}C_0\epsilon$. Thus
$\mathcal{T}=\infty$. And hence (\ref{result}) holds for all $t\in\Bbb R$.

 \section{Proof of scattering}

(\ref{result3}) is  in fact  scattering type result. This follows if one has shown
\begin{align}\label{ftma1}
\int^{\infty}_0\| O({\partial^2_x\phi }|\partial_x\phi|^2) \|_{H^{2}_x}ds\lesssim  1.
\end{align}
But (\ref{ftma1})  has been proved in Step 3 of Section 4.

\section{Appendix}

The following is  Hamilton's interpolation  inequality proved in [\cite{Hami},Section 12].
\begin{lemma}\label{Hami}
Let   $T$ be any Tensor defined on manifold $\Sigma$. For   $1\le j\le i-1$, there exists a constant $C$ depending
only on dimension of $\Sigma$ and $i$, which is independent of the metric $g$ and connection such that
\begin{align}
\int_{\Sigma} |\nabla^{j} T|^{\frac{2i}{j}}d\mu\le C\max_{\Sigma}|T|^{2(\frac{i}{j}-1)}\int_{\Sigma} |\nabla^{i} T|^{2}d\mu.
\end{align}

\end{lemma}

The following is linear dispersive estimates.

\begin{lemma}
Let   $1\le q\le 2$, and $f\in L^{q}(\Bbb R^d)$. Then there exists  a constant $C>0$  depending only on $q,d$ such that
\begin{align}
\|e^{i\Delta t}f\|_{L^{q'}_x}\le C t^{\frac{d}{2}(1-\frac{2}{q})}\|f\|_{L^q_x},
\end{align}
where $q'=\frac{q}{q-1}$.
\end{lemma}

\section*{Acknowledgement}
The author owes sincere gratitude  to the referees for the insightful comments which have deeply improved the presentation of this work.
The author thanks Prof. Chong Song and Youde Wang for drawing the author's attention to SMCF and pointing out an error in the first version of this manuscript.
 This work is partially supported by NSF-China Grant-1200010237 and Grant-11631007.

\end{document}